\documentclass[11pt,reqno]{amsart}
\usepackage[letterpaper,margin=1in,footskip=0.25in]{geometry}
\usepackage{microtype}
\usepackage{amssymb}
\usepackage[mathic=true]{mathtools}
\usepackage{enumitem}
\usepackage[noadjust]{cite}
\renewcommand{\citepunct}{;\ }

\PassOptionsToPackage{dvipsnames}{xcolor}
\usepackage{tikz-cd}

\PassOptionsToPackage{pdfusetitle,colorlinks,pagebackref,bookmarksdepth=3}{hyperref}
\usepackage{bookmark}
\hypersetup{citecolor=[HTML]{2E7E2A},linkcolor=[HTML]{800006},urlcolor=[HTML]{8A0087}}
\usepackage[hyphenbreaks]{breakurl}

\newtheorem{theorem}{Theorem}[section]

\newtheorem{proposition}[theorem]{Proposition}
\newtheorem{question}[theorem]{Question}

\newtheorem{alphthm}{Theorem}

\theoremstyle{definition}
\newtheorem{definition}[theorem]{Definition}
\newtheorem{example}[theorem]{Example}

\theoremstyle{remark}
\newtheorem{remark}[theorem]{Remark}

\newtheoremstyle{cited}{.5\baselineskip\@plus.2\baselineskip\@minus.2\baselineskip}{.5\baselineskip\@plus.2\baselineskip\@minus.2\baselineskip}{\itshape}{}{\bfseries}{\bfseries .}{5pt plus 1pt minus 1pt}{\thmname{#1}\thmnumber{ #2}\thmnote{ \normalfont#3}}
\theoremstyle{cited}
\newtheorem{citedprop}[theorem]{Proposition}

\newtheoremstyle{citeddef}{.5\baselineskip\@plus.2\baselineskip\@minus.2\baselineskip}{.5\baselineskip\@plus.2\baselineskip\@minus.2\baselineskip}{}{}{\bfseries}{\bfseries .}{5pt plus 1pt minus 1pt}{\thmname{#1}\thmnumber{ #2}\thmnote{ \normalfont#3}}
\theoremstyle{citeddef}
\newtheorem{citeddef}[theorem]{Definition}

\newtheoremstyle{step}{.25\baselineskip\@plus.1\baselineskip\@minus.1\baselineskip}{.25\baselineskip\@plus.1\baselineskip\@minus.1\baselineskip}{\itshape}{}{\bfseries}{\bfseries .}{5pt plus 1pt minus 1pt}{\thmname{#1}\thmnumber{ #2}\thmnote{ \normalfont(#3)}}
\theoremstyle{step}
\newtheorem{step}{Step}[subsection]

\DeclareMathOperator{\Spec}{Spec}

\newcommand{\bA}{\mathbf{A}}
\newcommand{\FF}{\mathbf{F}}
\newcommand{\QQ}{\mathbf{Q}}
\newcommand{\cO}{\mathcal{O}}
\newcommand{\fm}{\mathfrak{m}}
\newcommand{\fn}{\mathfrak{n}}
\newcommand{\fp}{\mathfrak{p}}
\newcommand{\fq}{\mathfrak{q}}

\newcommand{\id}{\mathrm{id}}

\newcommand{\hooklongrightarrow}{\lhook\joinrel\longrightarrow}

\begin{document}
\title{Equidimensional morphisms onto splinters are pure}
\author{Takumi Murayama}
\address{Department of Mathematics\\Purdue University\\150 N. University
Street\\West Lafayette, IN 47907-2067\\USA}
\email{\href{mailto:murayama@purdue.edu}{murayama@purdue.edu}}
\urladdr{\url{https://www.math.purdue.edu/~murayama/}}



\makeatletter
  \hypersetup{
    pdfsubject=\@subjclass,
    pdfkeywords=\@keywords
  }
\makeatother

\begin{abstract}
  We prove that a Noetherian ring $R$ is a splinter if and only if
  for every equidimensional surjective
  morphism $\operatorname{Spec}(S) \to \operatorname{Spec}(R)$, the map $R
  \to S$ is pure.
  This yields a large, nontrivial class of ring maps that are automatically
  pure.
  More generally, we prove that a locally Noetherian scheme $Y$ is locally
  a splinter if
  and only if every locally equidimensional morphism $X \to Y$ is
  strongly pure.
  Special cases of
  our results show that
  equidimensional fibrations over normal
  $\mathbf{Q}$-schemes or regular schemes of arbitrary characteristic
  are strongly pure.
  The main ingredient is a new factorization result for locally
  equidimensional morphisms of schemes, which is of independent interest.
  Additionally, we prove a weak Boutot-type theorem for
  $F$-rationality, which says that
  $F$-rationality descends under pure ring maps that are
  locally equidimensional under universally catenary assumptions.
  This statement is false without the locally equidimensional hypothesis.
\end{abstract}

\maketitle

\section{Introduction}\label{sect:intro}
\subsection{Motivation}
Let \(R \to S\) be a \textsl{pure} ring map, which means that \(M
\to M \otimes_R S\) is injective for every \(R\)-module \(M\)
\cite{Coh59,Oli70}.
There are many examples of pure ring maps, including inclusions of rings of
invariants by linearly reductive group actions, split maps (i.e., maps with a
left inverse as a map of \(R\)-modules), and faithfully flat maps.\medskip
\par It is very useful to know when a map \(R \to S\) is pure.
One reason is because a ring map is pure if and only if it is an effective
descent morphism for modules \cite{Oli71,JT84,Mes00}.
Another reason is that many properties of rings descend under pure ring maps.
For example, Noetherianity \cite[Corollaire 5.5]{Oli73},
normality \cite[Proposition 6.15\((b)\)]{HR74},
and the splinter property
\cite[Proposition 2.11(1)]{DT23} descend under pure maps.
For \(\QQ\)-algebras, rational singularities \citeleft\citen{Bou87}\citepunct
\citen{Mur25}\citemid Theorem C\citeright,
klt type singularities \cite{Zhu24},
Du Bois singularities \cite{GM}, and singularities of dense \(F\)-pure type
\cite{Yam} descend under pure maps.
For \(\FF_p\)-algebras, \(F\)-purity \cite[Proposition 5.13]{HR76}
and all variants of \(F\)-regularity \citeleft\citen{HH90}\citemid Proposition
4.12\citepunct \citen{Has10}\citemid Lemma 3.17\citeright\ descend
under pure maps.
Such results are sometimes called \textsl{Boutot-type theorems}.
See \cite[\S2]{HR74}, \cite{Wat97}, and Remark \ref{rem:termcanpure} for
some properties that do not satisfy a Boutot-type theorem.
\subsection{Main results}
\par The main goal of this paper is to prove the following result, which yields 
a large, nontrivial class of ring maps that are automatically pure.
For the definition of a locally equidimensional morphism, see Definition
\ref{def:equidimensionalmor}.
For example, finite type extensions of domains whose fibers all have the same
dimension are locally equidimensional.
Previously, Theorem \ref{thm:mainaffine} was only known in equal characteristic
zero \citeleft\citen{CGM16}\citemid Lemma 2.7\citepunct \citen{Zhu24}\citemid
Proof of Corollary A.2\citeright.
\begin{alphthm}\label{thm:mainaffine}
  Let \(R\) be a Noetherian ring.
  Then, \(R\) is a splinter if and only if for every finite type ring map
  \(\varphi\colon R \to S\) such that
  \(\Spec(\varphi)\) is locally equidimensional and surjective, the map
  \(\varphi\) is pure.
\end{alphthm}
\par A Noetherian ring \(R\) is a
\textsl{splinter} if for every module-finite ring
map \(\varphi\colon R \to S\) such that \(\Spec(\varphi)\) is surjective, the
map \(\varphi\)
splits, i.e., \(\varphi\)
has a left inverse as a map of
\(R\)-modules
\citeleft\citen{Ma88}\citemid p.\ 178\citepunct
\citen{Bha12}\citemid Definition 1.2\citeright.
There are many examples of splinters:
All regular rings are splinters (the direct summand theorem
\citeleft\citen{Hoc73}\citemid Theorem
2\citepunct \citen{And18b}\citemid Th\'eor\`eme 0.2.1\citeright),
a \(\QQ\)-algebra is a splinter if and only if it is normal \cite[Lemma
2]{Hoc73}, and
all weakly \(F\)-regular \(\FF_p\)-algebras are splinters
\cite[Theorem 5.25]{HH94jag}.\medskip
\par Theorem \ref{thm:mainaffine} shows that
despite the definition of a splinter only mentioning module-finite
maps, the splinter property has strong consequences for many ring maps
with positive-dimensional fibers.\medskip
\par More generally, we show the following scheme-theoretic version
of Theorem
\ref{thm:mainaffine}.
Even in the situation of Theorem \ref{thm:mainaffine}, the result below provides
additional information on the purity of maps of the form \(R_{\fq \cap R} \to
S_\fq\) for prime ideals
\(\fq \subseteq S\).
\begin{alphthm}\label{thm:splinterequidim}
  Let \(Y\) be a locally Noetherian scheme.
  Then, \(\cO_{Y,y}\) is a splinter for every \(y \in Y\)
  if and only if
  every locally equidimensional morphism \(f\colon X \to
  Y\) is strongly pure.
\end{alphthm}
Following \cite{CGM16},
a morphism \(f\colon X \to Y\) of schemes is \textsl{strongly pure} if
for every \(x \in X\), the map \(\cO_{Y,f(x)} \to \cO_{X,x}\) is pure.
See Definition \ref{def:partiallypure}.\medskip
\par One advantage of strongly pure morphisms over pure ring maps is that some
properties of rings are known to descend
under strongly pure morphisms, even though they do not descend under pure ring
maps.
In other words, a weak Boutot-type theorem asking for descent under strongly
pure maps can hold for a property of rings even though the full Boutot-type
theorem fails for it.
For example, \(F\)-injectivity
descends under strongly pure morphisms
\citeleft\citen{Mur21}\citemid Lemma A.3\citepunct \citen{DM24}\citemid Theorem
3.8\citeright, even though \(F\)-injectivity does not descend under pure ring
maps in
general \cite{Wat97}.\medskip
\par In this direction, we
prove a weak Boutot-type theorem for \(F\)-rationality, which says that
\(F\)-rationality
descends under
locally equidimensional partially pure morphisms as long as the base ring is
universally catenary.
This illustrates another advantage of morphisms of the form in
Theorem \ref{thm:splinterequidim} over pure maps since
\(F\)-rationality does not descend under pure ring maps in
general \cite{Wat97}.
The universally catenary assumption holds for example when \(\cO_{Y,y}\) is
excellent or \(F\)-finite \cite[Theorem 2.5]{Kun76}.
\begin{alphthm}\label{thm:frationallesp}
  Let \(f\colon X \to Y\) be a locally equidimensional morphism between
  locally Noetherian schemes of equal characteristic \(p > 0\).
  Let \(y \in Y\) be a point such that \(f\) is partially pure at \(y\) and
  \(\cO_{Y,y}\) is universally catenary.
  If \(X\) is \(F\)-rational along \(f^{-1}(y)\), then \(\cO_{Y,y}\) is
  \(F\)-rational.
  \par In particular, if \(f\) is locally equidimensional and
  partially pure at every \(y \in Y\), \(X\) is \(F\)-rational, and
  \(Y\) is universally catenary, then \(Y\) is
  \(F\)-rational.
\end{alphthm}
Following \cite{CGM16},
a morphism \(f\colon X \to Y\) of schemes is \textsl{partially pure at \(y \in
Y\)} if there exists a point \(x \in f^{-1}(y)\) such that
the map \(\cO_{Y,y} \to \cO_{X,x}\) is pure.
See Definition \ref{def:partiallypure}.
Strongly pure surjective morphisms are partially pure, and pure ring maps induce
partially pure morphisms of affine schemes \cite[Lemma 2.2]{HH95}.\medskip
\par In light of Theorems \ref{thm:splinterequidim} and
\ref{thm:frationallesp}, we ask the following:
\begin{question}
  What classes of singularities satisfy a weak Boutot-type theorem?
  In other words, what classes of singularities descend under surjective
  morphisms that are
  \begin{enumerate}[label=\((\alph*)\)]
    \item locally equidimensional;
    \item either strongly pure or partially pure; or
    \item both locally equidimensional and either strongly pure or partially
      pure?
  \end{enumerate}
\end{question}
We mention one example that we have not mentioned before:
\(\QQ\)-factoriality descends under surjective
locally equidimensional locally projective morphisms of integral normal
Noetherian separated schemes \cite[Theorem 3.18]{Nak10}.
\subsection{An example}
\par Before describing the proofs of Theorems \ref{thm:mainaffine} and
\ref{thm:splinterequidim}, we note that
there are many examples of locally equidimensional morphisms
satisfying the hypotheses in Theorem \ref{thm:splinterequidim} that are not
faithfully flat.
Recall that faithfully flat morphisms \(f\colon X \to Y\)
locally of finite type
are locally equidimensional assuming that either \(X\) is irreducible and \(f\)
is dominant or that \(X\) and \(Y\) are locally equidimensional
\citeleft\citen{EGAIV3}\citemid Corollaire 14.2.2\citepunct
\citen{EGAIV4}\citemid Err\textsubscript{IV}, 36\citeright.
However, the converse is not true:
quotients by linearly reductive groups and finite morphisms are usually not
flat.\medskip
\par Theorem \ref{thm:splinterequidim} says in particular that if \(Y\) is a
locally Noetherian normal \(\QQ\)-scheme or a regular scheme of arbitrary
characteristic, then equidimensional fibrations \(X
\to Y\) are strongly pure.
Here, a surjective morphism \(f\colon X \to Y\) of integral
schemes is a \textsl{fibration} if \(K(Y)\) is algebraically closed
in \(K(X)\) \citeleft\citen{Fuj81}\citemid Definition 1.19\citepunct
\citen{Fuj82}\citemid \((1.19')\)\citeright.
\par
Fibrations are usually not flat.
We mention an explicit example of such a fibration that is nevertheless strongly
pure by Theorem \ref{thm:splinterequidim}.
\begin{example}\label{ex:gross}
  In \cite[Example 3.5]{Gro94}, Gross constructs an
  equidimensional elliptic fibration
  \begin{equation}\label{eq:gross}
    \bar{X}' \longrightarrow \bar{Y}_1
  \end{equation}
  between complex projective
  varieties such that \(\bar{X}'\) is a
  minimal Calabi--Yau threefold and \(\bar{Y}_1\) is a normal surface with
  \(A_1\) singularities.
  \par We claim that \eqref{eq:gross}
  is not flat.
  Since \(\bar{X}'\) is terminal, the singular
  locus of \(\bar{X}'\) is a finite set \cite[Proposition 0.4]{Rei83}.
  Since the fibers of \eqref{eq:gross}
  are equidimensional of dimension 1, we
  see that for every \(y \in \bar{Y}_1\), there exists a point
  \(x \in \bar{X}'\) mapping to \(y\) that is regular on \(\bar{X}'\).
  If \eqref{eq:gross} were flat, this would imply that \(\bar{Y}_1\)
  is regular by
  \cite[Corollaire 6.5.2\((i)\)]{EGAIV2}, a contradiction.
\end{example}
\begin{remark}\label{rem:termcanpure}
  \par By Theorem
  \ref{thm:splinterequidim}, the morphism
  \eqref{eq:gross} is strongly pure.
  Since a surface is terminal if and only if it is regular \cite[Remark
  0.3\((a)\)]{Rei83}, Example \ref{ex:gross} shows that terminal singularities do
  not descend under pure maps.
  Finite quotient singularities are also examples where
  terminal singularities and canonical singularities do not descend under pure maps
  \cite[Example 1.5\((ii)\)]{Rei80}.
\end{remark}

\subsection{Description of proof}\label{sect:describeproof}
\par To prove Theorem \ref{thm:splinterequidim}, we prove a new
factorization result for locally
equidimensional morphisms, which is of independent interest.
See Proposition \ref{prop:charequidimac} for the precise statement, which
strengthens \cite[Proposition 13.3.1]{EGAIV3}.
This factorization result
allows us to avoid the Bertini-type theorems used in the proof of
\cite[Lemma 2.7]{CGM16}.\medskip
\par To explain why our new factorization result is necessary, we
describe our approach for the forward direction of Theorem
\ref{thm:splinterequidim}.
If \(f\colon X \to Y\) is equidimensional at a point \(x \in X\), then
\cite[Proposition 13.3.1]{EGAIV3} says we can
find a neighborhood \(U\) of \(x\) such that \(f\rvert_U\) factors as
\begin{equation}\label{eq:introfactor}
  \begin{tikzcd}[column sep={between origins,3em}]
    U\vphantom{\bA^e_Y}
    \arrow{rr}{g}\arrow{dr}[swap,pos=0.25]{\vphantom{p}f\rvert_U} &
    & \bA^e_Y\arrow{dl}[pos=0.25]{\vphantom{f\rvert_U}p}\\
    & Y
  \end{tikzcd}
\end{equation}
where \(g\) is quasi-finite and \(p\) is the projection morphism.
Since \(p\) is faithfully flat, to show that \(\cO_{Y,y} \to \cO_{X,x}\) is
pure, it suffices to show that
\begin{equation}\label{eq:notnecmodfinmap}
  \cO_{\bA^e_Y,g(x)} \longrightarrow \cO_{U,x}
\end{equation}
is pure.
If \(\cO_{\bA^e_Y,g(x)}\) is a splinter and \eqref{eq:notnecmodfinmap}
is module-finite, then we
would be done by the definition of a splinter.
We can reduce to the case when \eqref{eq:notnecmodfinmap}
is module-finite by passing to
Henselizations since Henselization preserves the splinter property
\cite[Theorem B]{DT23}.
However,
to show that \(\cO_{\bA^e_Y,g(x)}\) is a splinter, we would need to know that
\(g(x)\) is the generic point of \(p^{-1}(y)\) to apply \cite[Theorem
1.2]{Lyu25}.
\par In Proposition \ref{prop:charequidimac}, we prove that \(g(x)\) is indeed
the generic point of \(p^{-1}(y)\).
This is the key new ingredient to proving Theorem \ref{thm:splinterequidim}.
Without this input, one would have to show that the splinter property is stable
under polynomial ring extensions.
This is true in equal characteristic 0
(because normality is) and in equal characteristic \(p > 0\) \cite[Theorem
1.1]{Lyu25}, but is an open question in general \cite[Question \(1'\)]{DT23}.

\subsection*{Acknowledgments}
We thank Farrah Yhee for helpful conversations.
We thank the referee for suggesting a simplification of our original proof of
Proposition \ref{prop:charequidimac} and pointing out some typos.

\subsection*{Conventions}
All rings are commutative with identity and all ring maps are unital.
A point \(x\) on a scheme
\(X\) is \textsl{maximal} if \(x\) is the generic point of one of the
irreducible components of \(X\) \cite[Chapitre 0, (2.1.1)]{EGAInew}.

\section{Preliminaries}\label{sect:prelim}
\subsection{Purity}
We recall the definition of a pure map of rings or modules.
Pure maps are also known as \textsl{universally injective} ring maps \cite[p.\
3]{RG71}, and a ring map is pure if and only if it is an effective descent
morphism for modules \citeleft\citen{Oli71}\citepunct \citen{JT84}\citemid
Chapter II, \S5, Theorem 3\citepunct \citen{Mes00}\citemid Theorem 4\citeright.
\begin{citeddef}[{\citeleft\citen{Coh59}\citemid p.\ 383\citepunct
  \citen{Oli70}\citemid D\'efinition 1.1\citeright}]
  Let \(R\) be a ring.
  A map \(\varphi\colon M \to M'\) of \(R\)-modules
  is \textsl{pure} if
  \[
    \id_N \otimes_R \varphi\colon
    N \otimes_R M \longrightarrow N \otimes_R M'
  \]
  is injective for every \(R\)-module \(N\).
  A ring map \(R \to S\) is \textsl{pure} if it is a pure map of \(R\)-modules.
\end{citeddef}
For morphisms of schemes, we have the following notions.
\begin{definition}\label{def:partiallypure}
  Let \(f\colon X \to Y\) be a morphism of schemes.
  \begin{enumerate}[label=\((\roman*)\),ref=\roman*]
    \item \cite[(2) on p.\ 38]{CGM16}
      Let \(x \in X\) be a point.
      We say that \(f\) is \textsl{pure at \(x\)} if the map
      \(\cO_{Y,f(x)} \rightarrow \cO_{X,x}\)
      is pure.
      We say that \(f\) is \textsl{strongly pure} if \(f\) is pure at every \(x
      \in X\).
    \item \cite[(3) on p.\ 38]{CGM16}
      Let \(y \in f(X)\) be a point.
      We say that \(f\) is \textsl{completely pure at \(y\)} if \(f\) is pure at
      every \(x \in f^{-1}(y)\).
    \item\label{def:yamaguchistar} \cite[Definition 5.7]{Yam}
      Let \(y \in Y\) be a point.
      We say that \(f\) \textsl{satisfies Yamaguchi's condition \((*)\) at
      \(y\)} if there exists a maximal point \(x \in f^{-1}(y)\) such that \(f\)
      is pure at \(x\).
    \item \cite[(4) on p.\ 38]{CGM16}
      Let \(y \in Y\) be a point.
      We say that \(f\) is \textsl{partially pure at \(y\)} if \(f\) is pure at
      some \(x \in f^{-1}(y)\).
  \end{enumerate}
\end{definition}
Fixing \(x \in X\) and setting \(y = f(x)\), we have the following implications.
The two vertical implications in the top left square hold by \cite[Chapter I,
\S3, no.\ 5, Proposition 9]{Bou72}.
The other implications hold by definition.
\begin{equation}\label{eq:pureimplications}
  \begin{tikzcd}[baseline=(star.base)]
    \text{\(f\) is faithfully flat} \rar[Rightarrow]
    \arrow[Rightarrow]{d}
    \arrow[phantom]{dr}
    & \text{\(f\) is flat at \(x\)} \arrow[Rightarrow]{d}
    \\
    |[alias=star]|
    \text{\(f\) is strongly pure}
     \rar[Rightarrow]
    & \text{\(f\) is completely pure at \(y\)} \rar[Rightarrow]
    \dar[Rightarrow]
    & \begin{tabular}{@{}c@{}}
      \text{\(f\) satisfies Yamaguchi's}\\\text{condition
      \((\hyperref[def:yamaguchistar]{*})\) at \(y\)}
    \end{tabular}
    \dar[Rightarrow]\\
    & \text{\(f\) is pure at \(x\)} \rar[Rightarrow]
    & \text{\(f\) is partially pure at \(y\)}
  \end{tikzcd}
\end{equation}
These implications are not reversible in general
\citeleft\citen{CGM16}\citemid Corollary 5.6.2\citepunct
\citen{Yam}\citemid Example 5.8\citeright.
\begin{remark}
  Let \(f\colon X \to Y\) be a morphism of schemes.
  We connect the properties of morphisms defined in
  Definition \ref{def:partiallypure} to the notion of purity defined in
  \cite[Definition 3.10]{Mes04a}.
  Below, ``universally \(\mathcal{P}\)'' means ``\(\mathcal{P}\) holds
  under arbitrary base change.''
  Fpqc morphisms are defined following
  \cite[Definition 2.34]{Vis05}.
  \[
    \begin{tikzcd}[column sep=-1em,row sep=large]
      \text{\(f\) is
      fpqc}\arrow[Rightarrow]{d}[swap]{\substack{\text{\cite[Thm.\ 2.55]{Vis05}}\\\text{\cite[Rem.\ 3.13]{Mes04a}}}}
      \arrow[Rightarrow]{rr}{\eqref{eq:pureimplications}}
      & &
      \begin{tabular}{@{}c@{}}
        every base change \(f'\colon X' \to Y'\)\\
        is partially pure at every \(y' \in Y'\)
      \end{tabular}
      \arrow[Rightarrow]{d}\\
      \begin{tabular}{@{}c@{}}
        \(f\) is universally a\\
        regular epimorphism
      \end{tabular}
      \arrow[Rightarrow,start anchor=south,end anchor=north west]{dr}{\text{\cite[Prop.\ 3.3]{Mes04a}}}
      & &
      \begin{tabular}{@{}c@{}}
        \(f\) is universally\\
        scheme-theoretically dominant
      \end{tabular}
      \arrow[Rightarrow,start anchor=south,end anchor=north east]{dl}[swap]{\text{\cite[Prop.\ 6.4]{Mes04b}}}\\
      & \begin{tabular}{@{}c@{}}
        \(f\) is pure in the sense\\
        of \cite[Definition 3.10]{Mes04a}
      \end{tabular}
      \arrow[Rightarrow,dashed,start anchor=north west,end anchor=south,shift
      left=5pt,bend left=30]{ul}{\substack{\text{\(f\) quasi-compact}\\\text{\cite[Cor.\
      5.10]{Mes04b}}}}
      \arrow[Rightarrow,dashed,start anchor=north east,end anchor=south,shift
        right=5pt,bend
      right=30]{ur}[swap]{\substack{\text{\(f\) quasi-compact}\\\text{\cite[Prop.\
      6.4]{Mes04b}}}}
    \end{tikzcd}
  \]
  To show the vertical implication in the top right corner, we note that if
  \(f\) is partially pure at a point \(y \in Y\), then the composition
  \(\cO_{Y,y} \to (f_*\cO_X)_y \to \cO_{X,x}\)
  is injective for some \(x \in f^{-1}(y)\).
  Thus, if \(f\) is partially pure at every \(y \in Y\), then \(f\) is
  scheme-theoretically dominant.
\end{remark}
\subsection{Locally equidimensional morphisms}
We recall the definition of a locally equidimensional morphism from
\cite[Err\textsubscript{IV}, 34 and 35]{EGAIV4}.
Locally equidimensional morphisms were called ``equidimensional morphisms'' in
\cite[\S13]{EGAIV3}.
\begin{citedprop}[{\cite[Proposition 13.3.1]{EGAIV3}}]\label{prop:charequidim}
  Let \(f\colon X \to Y\) be a morphism of schemes that is locally of finite
  type.
  Let \(x \in X\) be a point.
  Then, the following conditions are equivalent:
  \begin{enumerate}[label=\((\alph*)\),ref=\alph*]
    \item\label{prop:charequidima}
      There exist an integer \(e\) and an open neighborhood \(U\) of \(x\)
      such that the image of every irreducible component of \(U\) under \(f\)
      is dense in an irreducible component of \(Y\) and such that
      for every \(x' \in U\), the topological space
      \[
        U \cap f^{-1}\bigl(f(x')\bigr)
      \]
      is equidimensional of dimension \(e\).
    \item\label{prop:charequidimb}
      There exist an integer \(e\), an open neighborhood \(U\) of \(x\), and
      a factorization
      \[
        \begin{tikzcd}[column sep={between origins,3em}]
          U\vphantom{\bA^e_Y}
          \arrow{rr}{g}\arrow{dr}[swap,pos=0.25]{\vphantom{p}f\rvert_U} &
          & \bA^e_Y\arrow{dl}[pos=0.25]{\vphantom{f\rvert_U}p}\\
          & Y
        \end{tikzcd}
      \]
      where \(g\) is quasi-finite and \(p\) is the projection map,
      such that the image of every irreducible component of \(U\) under \(g\) is
      dense in \(\bA^e_Y\).
  \end{enumerate}
\end{citedprop}
\begin{citeddef}[{\citeleft\citen{EGAIV3}\citemid D\'efinition 13.2.2
  and D\'efinition 13.3.2\citepunct
  \citen{EGAIV4}\citemid Err\textsubscript{IV}, 34 and 35\citeright}]
  \label{def:equidimensionalmor}
  Let \(f\colon X \to Y\) be a morphism of schemes that is locally of finite
  type.
  Let \(x \in X\) be a point.
  \begin{enumerate}[label=\((\roman*)\)]
    \item Suppose that \(X\) and \(Y\) are irreducible and that \(f\) is
      dominant.
      Denote by \(\eta\) the generic point of \(Y\).
      We say that \(f\) is \textsl{equidimensional at \(x\)} if 
      \begin{equation}\label{eq:equidimensionalmor}
        \dim_x\bigl(f^{-1}\bigl(f(x)\bigr)\bigr) = \dim\bigl(f^{-1}(\eta)\bigr).
      \end{equation}
    \item In general,
      we say that \(f\) is \textsl{equidimensional at \(x\)} if one of the
      equivalent conditions in Proposition \ref{prop:charequidim} holds.
  \end{enumerate}
  The two definitions are equivalent for dominant morphisms locally of finite
  type between irreducible schemes by \cite[pp.\ 196--197]{EGAIV3}.
  \par We say that \(f\) is \textsl{locally equidimensional} if \(f\) is
  equidimensional at every \(x \in X\).
  We say that \(f\) is \textsl{equidimensional} if \(f\) is locally
  equidimensional and if the fibers \(f^{-1}(f(x))\) are all
  equidimensional.
\end{citeddef}

\section{Factoring locally equidimensional morphisms}\label{sect:factor}
We show that locally equidimensional morphisms have the following factorization
property.
This property is stronger than condition 
\((\ref{prop:charequidimb})\) of Proposition \ref{prop:charequidim}.
We thank the referee for suggesting a simplification of our original proof.
\begin{proposition}\label{prop:charequidimac}
  Let \(f\colon X \to Y\) be a morphism of schemes that is locally of finite
  type.
  Let \(x \in X\) be a point and set \(y = f(x)\).
  The morphism \(f\) is equidimensional at \(x\) if and only if the following
  condition holds:
  \begin{enumerate}[label=\((\alph*)\),start=3,ref=\alph*]
    \item\label{prop:charequidimc}
      There exists an integer \(e\), an open neighborhood \(U\) of \(x\), and
      a factorization
      \[
        \begin{tikzcd}[column sep={between origins,3em}]
          U\vphantom{\bA^e_Y}
          \arrow{rr}{g}\arrow{dr}[swap,pos=0.25]{\vphantom{p}f\rvert_U} &
          & \bA^e_Y\arrow{dl}[pos=0.25]{\vphantom{f\rvert_U}p}\\
          & Y
        \end{tikzcd}
      \]
      where
      \(g\) is quasi-finite and
      \(p\) is the projection map, such that
      \begin{enumerate}[label=\((\ref*{prop:charequidimc}\mathchar`-\roman*)\),
          ref=\roman*]
        \item the image of every irreducible component of \(U\) under \(g\) is
          dense in \(\bA^e_Y\), and
        \item\label{prop:charequidimacii}
          for every generic point \(x_0 \in f^{-1}(y)\) of an
          irreducible component of \(f^{-1}(y)\) containing \(x\), the morphism
          \(g\) maps \(x_0\) to the generic point of \(p^{-1}(y)\).
      \end{enumerate}
  \end{enumerate}
\end{proposition}
\begin{proof}
  Throughout this proof,
  \((\ref{prop:charequidima})\) and
  \((\ref{prop:charequidimb})\) refer to the condition in
  Proposition \ref{prop:charequidim}.
  The implication \((\ref{prop:charequidimc}) \Rightarrow
  (\ref{prop:charequidimb})\) is clear.
  It therefore suffices to show
  \((\ref{prop:charequidima})+(\ref{prop:charequidimb}) \Rightarrow
  (\ref{prop:charequidimc})\).
  Since the only difference between \((\ref{prop:charequidimb})\) and
  \((\ref{prop:charequidimc})\) is
  \((\ref{prop:charequidimc}\mathchar`-\ref{prop:charequidimacii})\), it
  suffices to show that any morphism satisfying
  \((\ref{prop:charequidima})\) and
  \((\ref{prop:charequidimb})\) satisfies
  \((\ref{prop:charequidimc}\mathchar`-\ref{prop:charequidimacii})\).
  \par Let \(x_0 \in f^{-1}(y)\) be a generic point of an
  irreducible component of \(f^{-1}(y)\) containing \(x\).
  Suppose that \(z_0 = g(x_0) \in p^{-1}(y)\) is not maximal.
  We know that \(g\) restricts to a quasi-finite morphism
  \[
    h\colon U \cap f^{-1}(y) \longrightarrow p^{-1}(y)
  \]
  by base change \cite[Proposition 6.11.5\((iii)\)]{EGAInew}.
  Note that \(p^{-1}(y) \cong \bA^e_{\kappa(y)}\) is integral of dimension
  \(e\).
  Since \(z_0\) is not maximal in \(\bA^e_{\kappa(y)}\), the closure
  \(\overline{\{z_0\}} \subseteq p^{-1}(y)\) is of dimension \(< e\) by
  \cite[Corollaire 4.1.2.1]{EGAIV2}.
  On the other hand, the restriction of \(h\) to \(\overline{\{x_0\}}
  \subseteq U \cap f^{-1}(y)\) factors as
  \[
    \overline{\{x_0\}} \longrightarrow \overline{\{z_0\}} \hooklongrightarrow
    p^{-1}(y)
  \]
  by \cite[Proposition 4.6.2]{EGAInew}.
  Since \(h\) is quasi-finite, we know that \(\overline{\{x_0\}} \to
  \overline{\{z_0\}}\) is quasi-finite.
  Since \(\dim(\overline{\{x_0\}}) = e\) by \((\ref{prop:charequidima})\),
  we must have \(\dim(\overline{\{z_0\}}) \ge e\)
  by \cite[Th\'eor\`eme 4.1.2]{EGAIV2}, a contradiction.
  Thus, \(z_0 = g(x_0)\) is maximal in \(p^{-1}(y)\), proving
  \((\ref{prop:charequidimc}\mathchar`-\ref{prop:charequidimacii})\).
\end{proof}

\section{Proofs of Theorems \texorpdfstring{\ref{thm:mainaffine} and
\ref{thm:splinterequidim}}{\ref*{thm:mainaffine} and
\ref*{thm:splinterequidim}}}
\label{sect:proofs}
We first prove the following local version of the forward direction of
Theorem \ref{thm:splinterequidim}.
\begin{theorem}\label{thm:splinterequidimlocal}
  Let \(f\colon X \to Y\) be a morphism locally of finite type between locally
  Noetherian schemes.
  Let \(x \in X\) be a point such that \(f\) is equidimensional at \(x\), and
  set \(y = f(x)\).
  If \(\cO_{Y,y}\) is a splinter, then the map
  \(\cO_{Y,y} \rightarrow \cO_{X,x}\) is pure.
\end{theorem}
As mentioned in \S\ref{sect:intro},
related statements in equal characteristic zero are proved in \cite[Lemma
2.7]{CGM16} and \cite[Proof of Corollary A.2]{Zhu24}.
In addition to Theorem \ref{thm:splinterequidim} applying more generally
to morphisms of
schemes in arbitrary characteristic, the key difference between Theorem
\ref{thm:splinterequidim}
and the results in \cite{CGM16} and \cite{Zhu24}
is that both \cite{CGM16} and \cite{Zhu24} state that \(f\) is partially pure,
while Theorem \ref{thm:splinterequidim} shows the stronger conclusion that \(f\)
is strongly pure.
In \cite{CGM16}, this is because the authors work with closed points \(x \in
f^{-1}(y)\) to apply Bertini-type theorems.
The proof of \cite{Zhu24} can be modified to show that \(f\) is strongly pure
(still assuming that \(Y\) is of equal characteristic zero).
\begin{proof}[Proof of Theorem \ref{thm:splinterequidimlocal}]
  By Proposition \ref{prop:charequidimac}, there exist an integer \(e\), a
  neighborhood \(U\) of \(x\), and a factorization
  \begin{equation}\label{eq:splinterequidimfactor}
    \begin{tikzcd}[column sep={between origins,3em}]
      U\vphantom{\bA^e_Y}
      \arrow{rr}{g}\arrow{dr}[swap,pos=0.25]{\vphantom{p}f\rvert_U} &
      & \bA^e_Y\arrow{dl}[pos=0.25]{\vphantom{f\rvert_U}p}\\
      & Y
    \end{tikzcd}
  \end{equation}
  where \(g\) is quasi-finite and \(p\) is the projection map, such that
  the image of every irreducible component of \(U\) under \(g\) is
  dense in \(\bA^e_Y\) and \(g\) maps every generic point of \(f^{-1}(y)\) to
  the generic point of
  \(p^{-1}(y)\).
  To show that \(\cO_{Y,y} \to \cO_{X,x}\) is pure, it suffices to show that the
  composition
  \[
    \cO_{Y,y} \longrightarrow \cO_{X,x} \longrightarrow \cO_{X,x_0}
  \]
  is pure.
  We may therefore replace \(x\) by \(x_0\) to assume that \(x\) is a maximal
  point in \(f^{-1}(y)\).\medskip
  \par We now proceed in steps.
  \begin{step}
    We may assume that \(X\) and \(Y\) are integral and affine.
  \end{step}
  After shrinking \(Y\) around \(y\), we may assume that \(Y = \Spec(A)\) is
  affine.
  Since \(\cO_{Y,y}\) is a domain \cite[Lemma 2.1]{DT23}, after possibly
  shrinking \(Y\) around \(y\) we may assume that \(Y\) is integral.
  Moreover, after shrinking \(U\) around \(x\), we may assume that
  \(U = \Spec(B)\) is affine as well.
  \par We claim we may replace \(U\) by one of its irreducible components \(U'\)
  containing \(x\) with its reduced scheme structure.
  Letting \(U'\) be such a component, we have a factorization
  \[
    \cO_{Y,y} \longrightarrow \cO_{X,x} \longrightarrow \cO_{U',x}.
  \]
  If we show that the composition is pure,
  then the map \(\cO_{Y,y} \rightarrow \cO_{X,x}\) is pure.
  We therefore replace \(U\) by one of its irreducible components \(U'\) to
  assume that \(U\) is integral, in which case \(U \to \bA^e_Y\) is
  dominant.
  Note that the assumptions on the factorization
  \eqref{eq:splinterequidimfactor} are not affected by these
  reductions.\medskip
  \par Now 
  let \(Z = \bA^e_Y = \Spec(A[T_1,T_2,\ldots,T_e])\) and set \(z = g(x)\).
  Let
  \[
    \cO_{Z,z} \longrightarrow \cO_{Z,z}^h
  \]
  be the Henselization map, which is
  faithfully flat \cite[Theorem 43.8]{Nag75}
  (see also \cite[Th\'eor\`eme 18.6.6\((iii)\)]{EGAIV4}).
  \begin{step}\label{step:splinterequidim2}
    The Henselization \(\cO_{Z,z}^h\) is a splinter.
    In particular, \(\cO_{Z,z}^h\) is a domain.
  \end{step}
  It suffices to show that \(\cO_{Z,z}^h\) is a splinter because splinters are
  domains \cite[Lemma 2.1]{DT23}.
  By \cite[Theorem B]{DT23}, it suffices to show that \(\cO_{Z,z}\) is a
  splinter.
  Let \(\fp \subseteq A\) be the prime ideal corresponding to \(y\) and
  let \(\fq \subseteq A[T_1,T_2,\ldots,T_e]\)
  be the prime ideal corresponding to \(z\).
  Then, the map \(\cO_{Y,y} \to \cO_{Z,z}\) can be written as
  \[
    \varphi\colon A_\fp \longrightarrow
    \bigl(A_\fp[T_1,T_2,\ldots,T_e]\bigr)_\fq
  \]
  which is a regular map.
  We first show that \(\fm_y
  \cdot \cO_{Z,z} = \fm_z\), that is,
  \begin{align*}
    \fp \cdot \bigl(A_\fp[T_1,T_2,\ldots,T_e]\bigr)_\fq &= \fq \cdot
    \bigl(A_\fp[T_1,T_2,\ldots,T_e]\bigr)_\fq.
  \intertext{It suffices to show this equality before localizing at \(\fq\),
  that is,}
    \fp \cdot \bigl(A_\fp[T_1,T_2,\ldots,T_e]\bigr) &= \fq \cdot
    \bigl(A_\fp[T_1,T_2,\ldots,T_e]\bigr).
  \end{align*}
  This follows from the fact that \(z = g(x)\) is the generic point of
  \(p^{-1}(y)\), which implies
  \[
    \fq \cdot \bigl(\kappa(\fp)[T_1,T_2,\ldots,T_e]\bigr) = (0).
  \]
  Finally, since \(\cO_{Y,y}\) is a splinter, \(\varphi\) is regular, and
  \(\fm_y \cdot \cO_{Z,z} = \fm_z\), we see that \(\cO_{Z,z}\) is a splinter by
  \cite[Theorem 1.3]{Lyu25}.\medskip
  \par We now consider the diagram
  \[
    \begin{tikzcd}
      U' \rar\dar[swap]{g'} & U_z \rar\dar & U\dar{g}\\
      \Spec\bigl(\cO_{Z,z}^h\bigr) \rar & \Spec(\cO_{Z,z}) \rar & Z
    \end{tikzcd}
  \]
  with Cartesian squares.
  \begin{step}\label{step:splinterequidim3}
    There exists a point \(x' \in U'\) such that \(x'\) maps to \(x \in U\).
    For any such point \(x'\), the map 
    \begin{equation}\label{eq:splinterequidimstep3}
      \cO_{Z,z}^h \longrightarrow \cO_{U',x'}
    \end{equation}
    is a module-finite extension.
  \end{step}
  Note that \(g'\) is quasi-finite by base change
  \cite[Proposition 6.11.5\((iii)\)]{EGAInew} and that
  \(g'\) is dominant by flat base change \cite[Proposition
  6.1.7\((ii)\)]{EGAInew}.
  Since the morphism \(U' \to U_z\) is faithfully flat by base
  change \cite[Corollaire 2.2.13\((i)\)]{EGAIV2} and \(x \in U_z\),
  there exists a point \(x' \in U'\) such that \(x'\)
  maps to \(x \in U\).
  \par Now given a point \(x' \in U'\) such that \(x'\) maps to \(x \in U\),
  consider the map \eqref{eq:splinterequidimstep3}.
  By \cite[Th\'eor\`eme 18.5.11\((c)\)]{EGAIV4}, since \(\cO_{Z,z}^h\) is
  Henselian and \(g'\) is quasi-finite and of finite type,
  the map \eqref{eq:splinterequidimstep3} is module-finite.
  Also, the map
  \eqref{eq:splinterequidimstep3} is injective because \(g'\) is dominant
  \cite[Proposition 4.4.6\((i)\)]{EGAInew} and \(\cO_{Z,z}^h\) is a domain by
  Step \ref{step:splinterequidim2}.\medskip
  \begin{step}
    Conclusion of proof.
  \end{step}
  \par Consider the commutative diagram
  \[
    \begin{tikzcd}[column sep={3em,between origins}]
      \cO_{U',x'} & & \arrow[hook']{ll} \cO^h_{Z,z}\\
      \cO_{U,x} \uar & & \arrow{ll} \cO_{Z,z} \uar\\
      & \cO_{Y,y} \arrow{ul}\arrow{ur}
    \end{tikzcd}
  \]
  of Noetherian local rings.
  Since \(\cO^h_{Z,z}\) is a splinter by Step \ref{step:splinterequidim2}
  and
  \begin{equation}\label{eq:modfinext}
    \cO^h_{Z,z} \hooklongrightarrow \cO_{U',x'}
  \end{equation}
  is a module-finite extension by Step \ref{step:splinterequidim3},
  the map \eqref{eq:modfinext} splits.
  In particular, the map \eqref{eq:modfinext} is pure.
  The composition \(\cO_{Y,y} \to \cO_{Z,z} \to \cO_{Z,z}^h\) is faithfully
  flat, and in particular, is pure by \cite[Chapter I, \S3, no.\ 
  5, Proposition 9]{Bou72}.
  By the commutativity of the diagram and the fact that pure ring maps are
  stable under composition \citeleft\citen{Oli70}\citemid Proposition
  1.2\((ii)\)\citepunct \citen{Oli73}\citemid Lemme 2.3\((i)\)\citeright,
  the composition
  \[
    \cO_{Y,y} \longrightarrow \cO_{X,x} \longrightarrow \cO_{U',x'}
  \]
  is pure.
  Thus,
  \(\cO_{Y,y} \rightarrow \cO_{X,x}\) is pure.
\end{proof}
We now prove Theorem \ref{thm:splinterequidim}.
Theorem \ref{thm:mainaffine} will then follow as a consequence.
\begin{proof}[Proof of Theorem \ref{thm:splinterequidim}]
  The implication \(\Rightarrow\) follows from Theorem
  \ref{thm:splinterequidimlocal}.\medskip
  \par It remains to show the implication \(\Leftarrow\).
  Since the splinter property is local \cite[Lemma 2.3]{DT23}, it suffices to
  show that for every affine open subset \(\Spec(R) \subseteq Y\), the ring
  \(R\) is a splinter.
  Let \(\varphi\colon R \to S\) be a module-finite ring map such that
  \(\Spec(\varphi)\) is surjective.
  We want to show that \(\varphi\) splits.
  The composition
  \[
    \Spec(S) \longrightarrow \Spec(R) \hooklongrightarrow Y 
  \]
  is locally equidimensional, and hence is strongly pure by assumption.
  Since \(\Spec(\varphi)\) is surjective and this composition is strongly pure,
  for every maximal ideal \(\fm \subseteq R\), there exists a
  prime ideal \(\fq \subseteq S\) such that \(R_\fm \to S_\fq\) is pure.
  This map factors as
  \[
    R_\fm \longrightarrow S \otimes_R R_\fm \longrightarrow S_\fq.
  \]
  Thus, \(R_\fm \to S \otimes_R R_\fm\) is pure.
  Since purity is a local condition \cite[Proposition 1.2\((vii)\)]{Oli70},
  this implies \(R \to S\) is pure.
  Since purity and
  splitting are equivalent for module-finite maps of Noetherian rings
  \cite[Chapitre I, Th\'eor\`eme 2.3]{Laz68}, this implies
  \(R \to S\) splits.
\end{proof}
\begin{proof}[Proof of Theorem \ref{thm:mainaffine}]
  Since the splinter property is local \cite[Lemma 2.3]{DT23}, 
  the implication \(\Rightarrow\) is a special case of the implication
  \(\Rightarrow\) of Theorem \ref{thm:splinterequidim}.\medskip
  \par The implication \(\Leftarrow\) follows from the fact that
  finite morphisms are locally equidimensional and the fact that purity and
  splitting are equivalent for module-finite maps of Noetherian rings
  \cite[Chapitre I, Th\'eor\`eme 2.3]{Laz68}.
\end{proof}

\section{A weak Boutot-type theorem for \texorpdfstring{\(F\)}{F}-rationality}
We recall the definition of \(F\)-rational rings.
\begin{citeddef}[{\citeleft\citen{FW89}\citemid Definition 1.10\citepunct
  \citen{HH94}\citemid \S2 and Definition 4.1\citeright}]
  Let \(R\) be a Noetherian ring.
  A sequence of elements \(x_1,x_2,\ldots,x_n \in R\) is a
  \textsl{sequence of parameters} if for every prime ideal \(\fp \subseteq R\)
  containing \(I = (x_1,x_2,\ldots,x_n)\), the elements
  \(\frac{x_1}{1},\frac{x_2}{1},\ldots,\frac{x_n}{1} \in R_\fp\) form part of a
  system of parameters in \(R_\fp\).
  In this case, we say that \(I\) is a \textsl{parameter ideal}.
  \par A Noetherian ring \(R\) of prime characteristic \(p > 0\) is
  \textsl{\(F\)-rational} if every parameter ideal \(I \subseteq R\)
  is tightly closed in the sense of \cite[Definition 3.1]{HH90}.
\end{citeddef}
For schemes, we adopt the following definition.
\begin{definition}
  Let \(X\) be a locally Noetherian scheme of prime characteristic \(p > 0\).
  For a point \(x \in X\), we say that \(X\) is \textsl{\(F\)-rational at \(x\)}
  if \(\cO_{X,x}\) is \(F\)-rational.
  We say that \(X\) is \textsl{\(F\)-rational} if \(X\) is \(F\)-rational at
  every \(x \in X\).
\end{definition}
\begin{remark}
  By \cite[Proof of Theorem 4.2\((e)\)]{HH94}, a Noetherian ring \(R\) of prime
  characteristic \(p > 0\) is \(F\)-rational if \(R_\fp\) is \(F\)-rational for
  every prime ideal \(\fp \subseteq R\).
  The converse is true if \(R\) is a quotient of a Cohen--Macaulay ring
  \cite[Theorem 4.2\((e)\)]{HH94}.
  This assumption holds for example when \(R\) is \(F\)-finite \cite[Remark
  13.6]{Gab04} or has a dualizing complex \cite[Corollary 1.4]{Kaw02}.
  As far as we are aware, whether \(F\)-rationality localizes is an open
  problem in general.
\end{remark}
We can now prove our weak Boutot-type theorem for \(F\)-rationality (Theorem
\ref{thm:frationallesp}).
\begin{proof}[Proof of Theorem \ref{thm:frationallesp}]
  The ``In particular'' statement follows from the first statement since
  strongly pure morphisms are completely pure at every point in the
  image.\medskip
  \par Let \(x_0\) be a point in \(f^{-1}(y)\) such that \(\cO_{Y,y} \to
  \cO_{X,x_0}\) is pure.
  Let \(x\) be a closed point in \(f^{-1}(y)\) such that \(x_0\) specializes to
  \(x\).
  Since \(\cO_{Y,y} \to \cO_{X,x_0}\) factors as
  \[
    \cO_{Y,y} \longrightarrow \cO_{X,x} \longrightarrow \cO_{X,x_0},
  \]
  we see that
  \[
    \varphi\colon \cO_{Y,y} \longrightarrow \cO_{X,x}
  \]
  is pure.
  Since \(\cO_{X,x}\) is normal \cite[Theorem 4.2\((b)\)]{HH94}, we see that
  \(\cO_{Y,y}\) is normal \cite[Proposition 6.15\((d)\)]{HR74}.
  Thus, \(\varphi\) is an injective map of domains.
  \par For the rest of the proof, we denote by \((R,\fm)\) and \((S,\fn)\) the
  local rings \(\cO_{Y,y}\) and \(\cO_{X,x}\), respectively.
  Let \(s_1,s_2,\ldots,s_n \in R\) be a sequence of parameters and
  set \(I = (s_1,s_2,\ldots,s_n)\).
  We claim that \(\varphi(s_1),\varphi(s_2),\ldots,\varphi(s_n)\) is a
  sequence of parameters and hence \(\varphi(I)S\) is a parameter ideal.
  Let \(\fq \subseteq S\) be a prime ideal containing \(\varphi(I)S\),
  and set \(\fp \coloneqq \varphi^{-1}(\fq) \subseteq R\).
  By definition, there exists a sequence of elements
  \(s'_{n+1},s'_{n+2},\ldots,s'_d \in R_\fp\)
  such that
  \[
    \frac{s_1}{1},\frac{s_2}{1},\ldots,\frac{s_n}{1},s_{n+1}',s'_{n+2},
    \ldots,s'_d \in R_\fp
  \]
  is a system of parameters for \(R_\fp\).
  After multiplying by units in \(R_\fp\),
  we may assume that the \(s_i'\) are images of elements \(s_i \in R\) for all
  \(i \in \{n+1,n+2,\ldots,d\}\).
  Now let \(t_1,t_2,\ldots,t_e \in S_\fq\) be a sequence of elements that map to
  a system of parameters in \(S_\fq \otimes_{R_\fp} \kappa(\fp)\).
  Since \(R_\fp\) is universally catenary and equidimensional
  and \(f\) is equidimensional at \(\fq\), we have the equality
  \[
    \dim(S_\fq) = \dim(R_\fp) + \dim\bigl(S_\fq \otimes_{R_\fp}
    \kappa(\fp)\bigr) = d+e
  \]
  by \cite[Proposition 13.2.3]{EGAIV3}.
  Since the sequence 
  \begin{equation}\label{eq:sop}
    \varphi(s_1),\varphi(s_2),\ldots,\varphi(s_d),t_1,t_2,\ldots,t_e \in
    S_\fq
  \end{equation}
  generates an \(\fq S_\fq\)-primary ideal and consists of
  \(\dim(S_\fq) = d+e\) elements,
  the sequence \eqref{eq:sop} is a system of parameters.
  Thus, \(\varphi(I)S\) is a parameter ideal.
  \par It remains to show that \(I\) is tightly closed.
  By \cite[Lemma 4.11]{HH90}, we have
  \[
    \varphi(I^*) \subseteq \bigl(\varphi(I)S\bigr)^* =
    \varphi(I)S.
  \]
  Since \(\varphi\) is pure, contracting back to \(R\) yields \(I^* =
  I\).
  Thus, \(R\) is \(F\)-rational.
\end{proof}

\end{document}